\documentclass[11pt]{amsart}
\usepackage[margin=1.05in]{geometry}
\usepackage{amsmath,amssymb,amsthm}
\usepackage{enumerate}
\usepackage[colorlinks=true,linkcolor=blue,citecolor=blue,urlcolor=blue]{hyperref}
\usepackage{xcolor}

\newtheorem{theorem}{Theorem}[section]
\newtheorem{proposition}[theorem]{Proposition}
\newtheorem{lemma}[theorem]{Lemma}
\newtheorem{corollary}[theorem]{Corollary}
\newtheorem*{maintheorem}{Main Theorem}
\theoremstyle{definition}
\newtheorem{definition}[theorem]{Definition}

\newtheorem{remark}[theorem]{Remark}
\numberwithin{equation}{section}

\newcommand{\Fp}{\mathbb{F}_p}
\newcommand{\CC}{\mathbb{C}}
\newcommand{\Vect}{\mathrm{Vec}}
\newcommand{\cZ}{\mathcal{Z}}
\newcommand{\cC}{\mathcal{C}}
\newcommand{\cD}{\mathcal{D}}
\newcommand{\cE}{\mathcal{E}}
\newcommand{\cB}{\mathcal{B}}
\newcommand{\one}{\mathbf{1}}
\newcommand{\Hom}{\operatorname{Hom}}
\newcommand{\End}{\operatorname{End}}
\newcommand{\id}{\mathrm{id}}
\newcommand{\Res}{\operatorname{Res}}
\newcommand{\Forg}{\operatorname{Forg}}
\newcommand{\Alt}{\operatorname{Alt}}
\newcommand{\RT}{\operatorname{RT}}
\newcommand{\Rep}{\operatorname{Rep}}
\newcommand{\Fun}{\operatorname{Fun}}
\newcommand{\Irr}{\operatorname{Irr}}
\newcommand{\Tr}{\operatorname{Tr}}

\newcommand{\sgn}{\operatorname{sgn}}
\newcommand{\Aut}{\operatorname{Aut}}
\newcommand{\spanop}{\operatorname{span}}

\begin{document}

\title[Inequivalent modular tensor categories with identical $S$, $T$ and $W$]{Inequivalent modular tensor categories\\ with identical $S$, $T$ and $W$}

\begin{abstract}
It remains an open question whether the modular data $(S,T)$ together with the Whitehead-link matrix $W$ determine a modular tensor category up to ribbon equivalence. We answer negatively by constructing two braided-inequivalent Dijkgraaf--Witten modular categories for $(\mathbb Z/p\mathbb Z)^3$ for prime $p\ge5$ with identical $(S,T,W)$ and, more generally, identical Reshetikhin--Turaev invariants of all framed oriented links with at most two components under a single bijection of simple objects. The construction varies an alternating $3$-cocycle while fixing a quadratic cohomology class, using restriction to subgroups of rank at most two to match link invariants and the quadratic and alternating classes together to obstruct braided equivalence. A complete classification of modular tensor categories up to ribbon equivalence therefore requires invariants beyond those of colored framed links with at most two components. We also analytically evaluate an uncolored partition function on a closed oriented three-manifold that separates all $(p-1)/2$ equivalence classes in the family.

\end{abstract}

\author{Ran Luo}
\address{School of Physics, Peking University, Beijing, China}
\email{ranluo@pku.edu.cn}

\author{Jiahua Tian}
\address{School of Physics, East China Normal University, Shanghai, China}
\email{jtian1905@gmail.com}

\maketitle

\tableofcontents

%======================================================================
\section{Introduction and the result}
%======================================================================

It was proved in~\cite{MS} that the modular data $(S,T)$ do not determine a modular tensor category (MTC) up to ribbon equivalence. The counterexamples arise from groups $G=\mathbb Z_q\rtimes\mathbb Z_p$ for odd primes $p,q$ with $p\mid q-1$. The $p$ Drinfeld centers $\cZ(\Vect_G^{\omega^u})$ are pairwise inequivalent as braided categories~\cite[Lemma~3.1]{MS}, yet realize at most three sets of modular data~\cite[Corollary~4.2]{MS}. To obtain information beyond $(S,T)$, the $W$-matrix was introduced in~\cite[\S2.1.2 and Definition~2.1]{BDGRTW} as a normalized invariant of a colored framed oriented Whitehead link. It distinguishes these examples for $|G|=55$ and hence is not determined by $(S,T)$~\cite[Theorem~4.1 and \S4]{BDGRTW}. This work left open the question of whether the combined data $(S,T,W)$ determine an MTC up to ribbon equivalence~\cite[Abstract and \S1]{BDGRTW}.

In this work, we prove that $(S,T,W)$ does \textbf{not} determine an MTC up to ribbon equivalence by establishing the following.

\begin{maintheorem}
For every prime $p\ge5$, put $\Fp=\mathbb Z/p\mathbb Z$. There exists a family of Dijkgraaf--Witten modular categories $\{\cC_u\}_{u\in\Fp^\times}$ associated with the group $(\mathbb Z/p\mathbb Z)^3$ such that the following statements hold for all $u,v\in\Fp^\times$.
\begin{enumerate}[(A)]
\item There is a bijection $\Lambda_{uv}:\Irr(\cC_u)\to\Irr(\cC_v)$ with $\Lambda_{uv}(\one)=\one$ such that for every framed oriented link $L\subset S^3$ with $n\in\{1,2\}$ components and all simple objects $X_1,\dots,X_n$ of $\cC_u$,
\begin{equation*}
\RT_{\cC_u}(L;X_1,\dots,X_n)=\RT_{\cC_v}(L;\Lambda_{uv}X_1,\dots,\Lambda_{uv}X_n).
\end{equation*}
Consequently the matrices $\widetilde S,S,T$ and the matrices $\widetilde W,W$ of \cite[\S2.1.2, Definition~2.1]{BDGRTW} of $\cC_u$ and $\cC_v$ coincide under $\Lambda_{uv}$.
\item If $\cC_u$ and $\cC_v$ are braided equivalent, then $u=\pm v$. In particular, if $u\neq\pm v$ then $\cC_u$ and $\cC_v$ are not ribbon equivalent.
\end{enumerate}
\end{maintheorem}

\begin{corollary}\label{cor:p5}
For $p=5$, the modular categories $\cC_1$ and $\cC_2$ \textup{(}twisted Drinfeld doubles of the abelian group $(\mathbb Z/5\mathbb Z)^3$ of order $125$\textup{)} have the same $S$, $T$ and $W$ under one $\Lambda_{12}$, but they are not ribbon equivalent.
\end{corollary}

\subsection{The family of categories}\label{subsec:notation}
We now construct the family in the Main Theorem. Throughout, $p\ge5$ is a prime, $\Fp=\mathbb Z/p\mathbb Z$, and
\begin{equation*}
V:=\Fp^{3},
\end{equation*}
regarded as the finite abelian group $(\mathbb Z/p\mathbb Z)^3$ whose subgroups are precisely its $\Fp$-linear subspaces with $\Fp$-linear group automorphisms. We write $\zeta:=e^{2\pi i/p}$ and $\eta:=e^{2\pi i/p^2}$, so that $\eta^p=\zeta$. The expression $\zeta^t$ is well defined for $t\in\Fp$. For $r\in\Fp$, let $[r]\in\{0,1,\dots,p-1\}\subset\mathbb Z$ be its integer representative. The failure of $[\,\cdot\,]:\Fp\to\mathbb Z$ to be additive is measured by the carry

\begin{equation*}
C_p(r,s):=\frac{[r]+[s]-[r+s]}{p}\in\{0,1\}\qquad(r,s\in\Fp),
\end{equation*}
so that $[r+s]=[r]+[s]-p\,C_p(r,s)\in\mathbb Z$. We retain the distinction between field elements and their integer representatives throughout. In particular, the numerator defining $C_p$ is computed in $\mathbb Z$, since it would vanish in $\Fp$. The exponent of the $p^2$-th root $\eta$ in the functions $g_a$ of Section~\ref{sec:cocycles} must likewise be computed as an integer without reduction modulo $p$.

For $a,b,c\in V$, let $\det(a,b,c)\in\Fp$ denote the determinant with columns $a,b,c$. We use the standard basis $\mathbf{e}_1,\mathbf{e}_2,\mathbf{e}_3$ and write

\begin{equation*}
Q(a):=a_1^2+a_2^2+a_3^2\in\Fp,\qquad V^\vee:=\Hom_{\Fp}(V,\Fp)
\end{equation*}
for $a=a_1\mathbf e_1+a_2\mathbf e_2+a_3\mathbf e_3, a_1,a_2,a_3\in\mathbb F_p$. Since $p\ge5$, both $2$ and $6$ are invertible in $\Fp$. Accordingly, $u/2$ and $u/6$ denote $u\cdot2^{-1}$ and $u\cdot6^{-1}$ in $\Fp$.

\begin{definition}\label{def:cocycles}
For $a,b,c\in V$ and $u\in\Fp$ put
\begin{equation*}
\kappa(a,b,c):=\zeta^{\sum_{i=1}^{3}[a_i]\,C_p(b_i,c_i)},\qquad
\tau_u(a,b,c):=\zeta^{\frac u6\det(a,b,c)},\qquad
\omega_u:=\kappa\,\tau_u .
\end{equation*}
\end{definition}

In the exponent of $\kappa$ the integer $\sum_i[a_i]C_p(b_i,c_i)$ may be reduced modulo $p$ which however does not affect the result since $\zeta^N=\zeta^{N\bmod p}$. By Lemma~\ref{lem:cocycle} these are normalized $3$-cocycles. Since $H^3(V,U(1))=H^4(V,\mathbb Z)=\operatorname{Sym}^2(V^\vee)\oplus\Lambda^3(V^\vee)$~\cite[Lemma~3.1]{JFT}, writing $\omega_u=\kappa\tau_u$ matches exactly this decomposition. Indeed, in Lemma~\ref{lem:invariants}, $\kappa$ and $\omega_u$ are shown to be related to quadratic cohomological invariant $q$ and alternating cohomological invariant $\Alt$, respectively.

We write
\begin{equation*}
\cC_u:=\cZ\big(\Vect_V^{\omega_u}\big),
\end{equation*}
equipped with the canonical ribbon structure of Section~\ref{subsec:center-conventions}.

%======================================================================
\subsection{Outline}\label{sec:outline}

\paragraph{\textbf{Equality of link invariants.}}
For $u\in\Fp^\times$, the simple objects of $\cC_u$ admit a common parametrization whose fluxes and projective characters are independent of $u$. Matching the same parameters defines the bijection $\Lambda_{uv}:{\rm Irr}(\cC_u)\to{\rm Irr}(\cC_v)$ of Definition~\ref{def:Lambda}. For a link with at most two components, the fluxes of its simple colors span a subspace $P\subseteq V$ of dimension at most two. Since the alternating factor $\tau_u$ is trivial on $P$, $\omega_u|_{P^3}=\omega_v|_{P^3}=\kappa|_{P^3}$. Restricting half-braidings for objects supported on $P$ therefore gives ribbon functors from either $(\cC_u)_P$ or $(\cC_v)_P$ to the common category $\cB_P=\cZ(\Vect_P^{\kappa|_{P^3}})$. Matched colors have the same projective multiplier and character after restriction, hence isomorphic images by Lemma~\ref{lem:match}. These ribbon functors preserve link evaluations, so all framed oriented link invariants with at most two components agree under $\Lambda_{uv}$ by Theorem~\ref{thm:step1}. This proves part~(A), including equality of $S,T,W$.

\smallskip
\paragraph{\textbf{Obstruction to braided equivalence.}}
Let $u,v\in\Fp^\times$. The electric simples are precisely the invertible objects of $\cC_u$. Every braided equivalence $\Phi:\cC_u\to\cC_v$ must therefore preserve the electric subcategories $\cE_u\simeq\Rep(V)$. Within $\cE_u$, the regular algebra $A_u=\Fun(V)$ is characterized up to algebra isomorphism as the unique connected \'etale algebra of dimension $|V|$. Hence $\Phi(A_u)\cong A_v$ as algebras by Lemma~\ref{lem:regular}. Taking right modules over $A_u$ recovers $\Vect_V^{\omega_u}$, so $\Phi$ induces a monoidal equivalence between $\Vect_V^{\omega_u}$ and $\Vect_V^{\omega_v}$ by Corollary~\ref{cor:pointedequiv}. By Lemma~\ref{lem:pointed}, this monoidal equivalence induces an automorphism $M\in GL_3(\Fp)$ satisfying $M^*\omega_v=\omega_u\,\partial\beta$. It must respect both the quadratic and alternating cohomological invariants, giving $Q(Ma)=Q(a)$ and $v\det M=u$ by Lemmas~\ref{lem:pointed} and~\ref{lem:invariants}. The first identity forces $\det M=\pm1$, while the second gives $u=\pm v$, proving Theorem~\ref{thm:step2}, namely $\cC_u\simeq_{\text{br}}\cC_v$ forces $u=\pm v$. This proves part (B).

We further prove Theorem~\ref{thm:observable} in which we convert the Borromean three-component amplitude into an uncolored surgery partition function that distinguishes the categories precisely modulo $u\mapsto-u$. This provides an independent proof of part (B).

In summary, our construction fixes the $\operatorname{Sym}^2(V^\vee)$ component, which controls all rank-$\le2$ restrictions and hence all one- and two-component RT link invariants, while varying the $\Lambda^3(V^\vee)$ component, which is invisible to such probes but remains detectable by the full braided category.

The work is organized as follows. In Section~\ref{sec:prelim}, we introduce the categorical conventions and link invariants. In Section~\ref{sec:cocycles}, we establish the cocycle identities and their behavior under restriction. In Section~\ref{sec:step1}, we prove equality of the link invariants. In Section~\ref{sec:step2}, we derive the obstruction to braided equivalence. In Section~\ref{sec:final}, we complete the proofs of the Main Theorem and its corollary. Finally, in Section~\ref{sec:observable}, we compute the closed three-manifold partition function and explain the additional topological response it detects. 

%======================================================================

\section{Categorical conventions and link invariants}\label{sec:prelim}

We use the standard conventions for $\CC$-linear fusion categories, spherical structures and centers in~\cite[\S2]{DGNO10}. We review only the choices and formulae needed below.

\subsection{Pointed categories and duality}\label{subsec:pointed-conventions}
For a finite abelian group $\Gamma$, let $\nu:\Gamma^3\to\CC^\times$ be a normalized $3$-cocycle, so that $\nu$ is one whenever an argument is zero and
\begin{equation}\label{eq:cocycle}
    \nu(y,z,w)\nu(x,y+z,w)\nu(x,y,z)=\nu(x+y,z,w)\nu(x,y,z+w).
\end{equation}
\begin{definition}\label{def:Vec}
The objects of $\Vect_\Gamma^\nu$ are finite-dimensional $\Gamma$-graded vector spaces $X=\bigoplus_{x\in\Gamma}X_x$, and the morphisms are degree-preserving linear maps. The tensor product is
\begin{equation*}
    (X\otimes Y)_z=\bigoplus_{x+y=z}X_x\otimes Y_y,
\end{equation*}
with the usual tensor product of linear maps. The unit is $\mathbf1=\mathbb C$ concentrated in degree $0$, with the standard unit constraints~\cite[Remark~2.6.3]{EGNO}. The constraint on associativity is
\begin{equation*}
    \alpha_{X,Y,Z}:(X\otimes Y)\otimes Z\longrightarrow X\otimes(Y\otimes Z),\qquad \alpha_{X,Y,Z}((v\otimes w)\otimes t) =\nu(x,y,z)\,v\otimes(w\otimes t),
\end{equation*}
for $v\in X_x$, $w\in Y_y$ and $t\in Z_z$.
\end{definition}

The simple objects of $\Vect_\Gamma^\nu$ are $\delta_x:=\CC 1_x, x\in\Gamma$, and we write $\iota_{x,y}:\delta_x\otimes\delta_y\xrightarrow{\sim}\delta_{x+y}$ for the map $1_x\otimes1_y\mapsto1_{x+y}$. We equip $\Vect_\Gamma^\nu$ with its standard rigid structure, in which the dual of a homogeneous object of degree $x$ has degree $-x$. Equivalently, for $X=\bigoplus_{x\in\Gamma}X_x$, $(X^*)_{-x}=(X_x)^*$. We further equip $\Vect_\Gamma^\nu$ with its unique positive spherical structure $\psi^{\Gamma,\nu}$, for which $d(\delta_x)=1$ since $\Vect_\Gamma^\nu$ is pointed~\cite[Example~2.26 and Remark~2.25]{DGNO10}~\cite[Proposition~8.23]{ENO}.

For every subgroup $P\le\Gamma$, the full fusion subcategory consisting
of objects supported on $P$ is canonically~\cite[Exercise~4.11.2]{EGNO}
\begin{equation*}
    \Vect_P^{\nu|_{P^3}}\subset\Vect_\Gamma^\nu.
\end{equation*}
Its monoidal and rigid structures are the restrictions of those on $\Vect_\Gamma^\nu$. Moreover, the restriction of $\psi^{\Gamma,\nu}$ to this subcategory is positive, and hence, by uniqueness of the positive spherical structure~\cite[Corollary~2.24 and Remark~2.25]{DGNO10},
\begin{equation*}
    \psi^{\Gamma,\nu}_X=\psi^{P,\nu|_{P^3}}_X
\end{equation*}
for every $P$-supported object $X$.

\subsection{Center and canonical ribbon structure}\label{subsec:center-conventions}
Let $\cC$ be a fusion category with associativity constraint $\alpha$
and the standard unit constraints, as for $\Vect_\Gamma^\nu$. We use the
following convention \cite[\S2.9]{DGNO10}
\cite[Definition~XIII.4.1]{Kas}.

\begin{definition}\label{def:center}
The Drinfel'd center of $\cC$ is the category $\cZ(\cC)$ with the following data.
Each object of $\cZ(\cC)$ is a pair $(X,c_{-,X})$, where $X\in\cC$ and $c_{V,X}:V\otimes X\to X\otimes V$ is a family of isomorphisms natural in $V\in\cC$, called a \emph{half-braiding}. It satisfies $c_{\one,X}=\id_X$ and, for all $V,W\in\cC$,
\begin{equation}\label{eq:halfbraiding}
    c_{V\otimes W,X}=\alpha_{X,V,W}(c_{V,X}\otimes\id_W)\alpha^{-1}_{V,X,W}(\id_V\otimes c_{W,X})\alpha_{V,W,X}.
\end{equation}
A morphism $(X,c_{-,X})\to(X',c_{-,X'})$ is a morphism $f:X\to X'$ in $\cC$ such that, for every $V$,
\begin{equation*}
    (f\otimes\id_V)c_{V,X}=c_{V,X'}(\id_V\otimes f).
\end{equation*}
The tensor product is $(X,c_{-,X})\otimes(Y,c_{-,Y})=(X\otimes Y,c_{-,X\otimes Y})$, where
\begin{equation}\label{eq:tensorhalf}
    c_{V,X\otimes Y}=\alpha^{-1}_{X,Y,V}(\id_X\otimes c_{V,Y})\alpha_{X,V,Y}(c_{V,X}\otimes\id_Y)\alpha^{-1}_{V,X,Y}.
\end{equation}
The unit is $(\one,\id)$, and the associativity and unit constraints are inherited from $\cC$. The braiding is
\begin{equation}\label{eq:braiding}
    c^{\cZ}_{(X,c_{-,X}),(Y,c_{-,Y})}=c_{X,Y}:X\otimes Y\longrightarrow Y\otimes X,
\end{equation}
using the half-braiding of the \emph{second} object evaluated at the first. The forgetful tensor functor $\Forg:\cZ(\cC)\to\cC$, called \emph{central functor} in~\cite[Definition~8.8.6]{EGNO}, sends $(X,c_{-,X})$ to $X$ and acts identically on underlying morphisms.
\end{definition}

The center $\cZ(\Vect_\Gamma^\nu)$ carries the canonical positive spherical structure induced from that of $\Vect_\Gamma^\nu$, and is modular \cite[\S2.9 and Remark~2.35]{DGNO10}. We use throughout the corresponding canonical ribbon twist
\begin{equation*}
    \theta_X=\upsilon_X^{-1}\psi^{\cZ}_X,
\end{equation*}
where $\upsilon_X:X\to X^{**}$ is the Drinfeld isomorphism~\cite[\S2.8.2, equations~(12)--(16)]{DGNO10}.

For an abelian group $\Gamma$, the structure of $\cZ(\Vect_\Gamma^\nu)$ is given by the following proposition.
\begin{proposition}[{\cite[\S4]{NN}, \cite[\S2.9 and Remark~4.65]{DGNO10}}]
\label{prop:centerabelian}
Let $\Gamma$ be a finite abelian group and $\nu$ a normalized $3$-cocycle. The center $\cZ(\Vect_\Gamma^\nu)$ has the following standard description. For each $a\in\Gamma$, the function
\begin{equation}\label{eq:multiplier}
    \gamma^\nu_a(x,y)=\frac{\nu(x,a,y)}{\nu(a,x,y)\nu(x,y,a)}.
\end{equation}
is a normalized $2$-cocycle.
\begin{enumerate}[(i)]
\item Objects are equivalently graded vector spaces $X=\bigoplus_aX_a$ with a $\gamma_a^\nu$-projective action on each $X_a$: $\rho_X(0)=\id$ and $\rho_X(x)\rho_X(y)|_{X_a}=\gamma_a^\nu(x,y)\rho_X(x+y)|_{X_a}$. Morphisms are graded intertwiners, and each $X_a$ is a center subobject. The simple objects are therefore indexed by pairs $(a,\pi)$, where $\pi$ is an irreducible $\gamma_a^\nu$-projective representation of $\Gamma$. The positive categorical dimension of $X$ is $d(X)=\dim_{\CC}X$; in particular, an object is invertible if and only if its underlying vector space is one-dimensional.

\item The corresponding half-braiding is
\begin{equation}\label{eq:halfexplicit}
    c_{U,X}(w\otimes v)=\rho_X(x)v\otimes w, \qquad w\in U_x,\quad v\in X.
\end{equation}
The tensor product is determined by~\eqref{eq:tensorhalf}, and the braiding is $c^{\cZ}(v\otimes w)=\rho_Y(a)w\otimes v$ for $v\in X_a$, $w\in Y$.

\item The full zero-flux subcategory is
$\cE\cong\Rep(\Gamma)$ as a braided monoidal category, whose tensor action is $\rho_{X\otimes Y}(x)=\rho_X(x)\otimes\rho_Y(x)$, associator is the identity, and braiding is the flip.
\end{enumerate}
\end{proposition}

The multiplier~\eqref{eq:multiplier} is the inverse of the multiplier in \cite[\S4, equations~(19)--(21) and Remark~4.1]{NN}, where $\rho(x+y)=\beta_a(x,y)\rho(x)\rho(y)$; this accounts for our associator and half-braiding conventions. The dimension statement uses~\cite[Remark~2.35(i)]{DGNO10} and the positive spherical structure fixed in Section~\ref{subsec:center-conventions}.
%======================================================================

\subsection{Link evaluations and functoriality}\label{subsec:RT-conventions}
For a $\CC$-linear ribbon fusion category $\cD$, a framed oriented $n$-component link $L$, and colors $X_1,\ldots,X_n\in\cD$, write
\begin{equation*}
    \RT_\cD(L;X_1,\ldots,X_n)\in\End_\cD(\mathbf1)\cong\CC
\end{equation*}
for its Reshetikhin--Turaev evaluation.  Replacing the color of any component by an isomorphic object does not change the evaluation~\cite[Chapter~I, \S2.9, Exercise~1]{Tu}. We suppress associators and unit constraints from the notation when they are determined by coherence. The ribbon-graph construction is functorial under $\CC$-linear ribbon functors. In particular, if $G:\cC\to\cD$ is such a strict ribbon functor in the sense of~\cite[\S2.4]{KT}, then
\begin{equation*}
    \RT_\cD(L;GX_1,\ldots,GX_n)=\RT_\cC(L;X_1,\ldots,X_n)
\end{equation*}
for every colored framed oriented link $L$.

\subsection{\texorpdfstring{Normalizations of $S,T,W$}{Normalizations of S,T,W}}
\label{subsec:STW-conventions}
For simple colors $X_i\in \operatorname{Irr}(\cD)$, we choose
\begin{gather*}
    D=\Big(\sum_Xd(X)^2\Big)^{1/2}>0,\qquad d(X)=\RT(O^0;X),\qquad \theta_X=\frac{\RT(O^{+1};X)}{d(X)},\\
    \widetilde S_{XY}=\RT(\mathrm{Hopf};X,Y),\qquad S=D^{-1}\widetilde S,\qquad T=\operatorname{diag}(\theta_X),\qquad  W_{XY}=\frac{\theta_X}{\theta_Y}\widetilde W_{XY}.
\end{gather*}
Here $O^m$ is the $m$-framed unknot, and $\widetilde W$ evaluates the specific framed oriented Whitehead diagram of~\cite[\S2.1.2 and Definition~2.1]{BDGRTW}. Consequently, a single vacuum-preserving bijection under which all one- and two-component link evaluations agree makes $d,\theta,D$ and all five matrices $\widetilde S,S,T,\widetilde W,W$ agree under the same bijection.

%======================================================================

\section{Cocycles and restriction to subgroups}\label{sec:cocycles}
%======================================================================

We show that the $u$-dependence of $\omega_u$ disappears after restriction to any subgroup $P\leq V$ of rank at most two.  This is the key input for comparing all one- and two-component link invariants in a common restricted ribbon category.

\begin{lemma}\label{lem:cocycle}
For every $u\in\Fp$, $\kappa$ and $\tau_u$ (hence $\omega_u$) are normalized $3$-cocycles on $V$.
\end{lemma}

\begin{proof}
Normalization holds because $[0]=0$, $C_p(0,s)=C_p(r,0)=0$ and $\det$ vanishes when an argument is $0$. Writing $\nu=\zeta^{f}$ with $f:V^3\to\Fp$, we can express~\eqref{eq:cocycle} as
\begin{equation*}
    (\delta f)(x,y,z,w):=f(y,z,w)-f(x+y,z,w)+f(x,y+z,w)-f(x,y,z+w)+f(x,y,z)=0\ \text{ in }\Fp .
\end{equation*}
For $\tau_u$, $f=\frac u6\det$ is trilinear, and expanding the second, third and fourth terms by linearity shows that the eight resulting terms cancel in pairs.
For $\kappa$, $f=\sum_if_i$ with $f_i(x,y,z)=[x_i]C_p(y_i,z_i)$, and it suffices to treat one $f_i$. For $b,c,d\in\Fp$ both $C_p(b,c)+C_p(b+c,d)$ and $C_p(c,d)+C_p(b,c+d)$ equal $([b]+[c]+[d]-[b+c+d])/p$. Hence, for $a,b,c,d\in\Fp$ and $F(a,b,c)=[a]C_p(b,c)$,
\begin{align*}
(\delta F)(a,b,c,d)&=([b]-[a+b])C_p(c,d)+[a]\big(C_p(b+c,d)-C_p(b,c+d)+C_p(b,c)\big)\\
&=([a]+[b]-[a+b])\,C_p(c,d)=p\,C_p(a,b)\,C_p(c,d)\equiv0\pmod p. \qedhere
\end{align*}
\end{proof}

For $a\in V$ define $g_a:V\to\CC^\times$ by
\begin{equation*}
g_a(x):=\eta^{-\sum_{i=1}^3[a_i][x_i]} \ ,
\end{equation*}
notice that here the integer exponent is \emph{not} reduced modulo $p$ but modulo $p^2$, and for any $g:V\to\CC^\times$ we define $(\partial g)(x,y):=g(x)g(y)g(x+y)^{-1}$.

\begin{lemma}\label{lem:multiplier}
For $u\in\Fp$ and $a,x,y\in V$, the multiplier~\eqref{eq:multiplier} can be written as
\begin{equation*}
    \gamma^{\omega_u}_a(x,y)=\zeta^{-\sum_i[a_i]C_p(x_i,y_i)-\frac u2\det(a,x,y)}=(\partial g_a)(x,y)\cdot b^u_a(x,y),\qquad b^u_a(x,y):=\zeta^{-\frac u2\det(a,x,y)} .
\end{equation*}
\end{lemma}

\begin{proof}
By \eqref{eq:multiplier}, the contribution of $\kappa$ to the exponent of $\zeta$ is
\begin{equation*}
    \sum_i\big([x_i]C_p(a_i,y_i)-[a_i]C_p(x_i,y_i)-[x_i]C_p(y_i,a_i)\big)=-\sum_i[a_i]C_p(x_i,y_i),
\end{equation*}
because $C_p$ is symmetric. For the $\tau_u$ factor, $\det(x,a,y)=-\det(a,x,y)$ and $\det(x,y,a)=\det(a,x,y)$ give
\begin{equation*}
    \frac u6\big(\det(x,a,y)-\det(a,x,y)-\det(x,y,a)\big)=\frac u6(-3)\det(a,x,y)=-\frac u2\det(a,x,y).
\end{equation*}
Finally, the definition of the carry gives
\begin{equation*}
    (\partial g_a)(x,y)=\eta^{-\sum_i[a_i]([x_i]+[y_i]-[x_i+y_i])}=\eta^{-p\sum_i[a_i]C_p(x_i,y_i)}=\zeta^{-\sum_i[a_i]C_p(x_i,y_i)}.
\end{equation*}
\end{proof}

\begin{lemma}\label{lem:rank2}
Let $P\leq V$ be a subgroup with $\dim_{\Fp}P\leq2$. For every $u\in\Fp$, one has $\det(x,y,z)=0, \tau_u(x,y,z)=1$ for all $x,y,z\in P$, and hence $\omega_u|_{P^3}=\kappa|_{P^3}$. Moreover, $\gamma^{\omega_u}_a(x,y)=(\partial g_a)(x,y)$ for all $a,x,y\in P$, thus $\Vect_P^{\omega_u|_{P^3}}=\Vect_P^{\kappa|_{P^3}}$ and $\cZ\!\left(\Vect_P^{\omega_u|_{P^3}}\right)=\cZ\!\left(\Vect_P^{\kappa|_{P^3}}\right)$ as equalities of categories for all $u$.
\end{lemma}

\begin{proof}
The first half is trivial, whereas the formula for the multiplier follows from Lemma~\ref{lem:multiplier} and $\det(a,x,y)=0$.
\end{proof}

%======================================================================
\section{Equality of link invariants}\label{sec:step1}
%======================================================================

\subsection{Simple objects and a common labeling}

The description of the center in Proposition~\ref{prop:centerabelian} reduces the classification of simple objects to projective representation theory. We begin with the semisimplicity argument needed for this classification.

\begin{lemma}\label{lem:maschke}
Let $H$ be a finite abelian group and $b$ a normalized $2$-cocycle on $H$ with values in roots of unity. The twisted group algebra $\CC^b[H]$, with basis $U_h$ and $U_gU_h=b(g,h)U_{g+h}$, is semisimple. Equivalently, every finite-dimensional $b$-projective representation of $H$, viewed as a $\CC^b[H]$-module via $U_h\mapsto\rho(h)$, is completely reducible.
\end{lemma}

\begin{proof}
Let $Z\subset\CC^\times$ be the finite subgroup generated by the values of $b$. By the cocycle identity and normalization, $H'=Z\times H$ with multiplication $(\lambda,g)(\lambda',h)=(\lambda\lambda'\,b(g,h),\,g+h)$ is a finite group, and $(\lambda,h)\mapsto\lambda U_h$ extends linearly to a surjective algebra homomorphism $\CC[H']\to\CC^b[H]$. Hence every finite-dimensional $\CC^b[H]$-module is a $\CC[H']$-module with the same submodules. Since $\CC[H']$ is semisimple by Maschke's theorem~\cite[Theorem~4.1.1(i)]{EGHLLSVY}, such modules are completely reducible, and $\CC^b[H]$ is semisimple, both by~\cite[Proposition~3.5.8]{EGHLLSVY}.
\end{proof}

\begin{proposition}\label{prop:simples}
Let $u\in\Fp^\times$. For $\xi\in V^\vee$ let $E_\xi$ be $\CC$ in degree $0$ with $\rho(x)=\zeta^{\xi(x)}$. Then $E_\xi$ is a simple object of $\cC_u$ and $E_0=\one$. For nonzero flux, let $a\in V\setminus\{0\}$ and $k\in\Fp$ and choose $r,s\in V$ with $\det(a,r,s)=1$. Every $x\in V$ is uniquely $x=t_xa+m_xr+n_xs$ with $t_x=\det(x,r,s)$, $m_x=\det(a,x,s)$, $n_x=\det(a,r,x)$. On $\CC^p$ with basis $(e_j)_{j\in\Fp}$ put
\begin{equation*}
    \pi_{a,k}(x)e_j:=\zeta^{kt_x+\frac u2m_xn_x+un_xj}\,e_{j+m_x}.
\end{equation*}
Then $\pi_{a,k}$ is a $b^u_a$-projective representation, $\rho:=g_a\,\pi_{a,k}$ is a $\gamma^{\omega_u}_a$-projective representation, and
\begin{equation*}
    X_u(a,k):=\big(\CC^p\text{ placed in degree }a,\ \rho\big)
\end{equation*}
is a simple object of $\cC_u$ with $\rho(a)=g_a(a)\zeta^k\,\id$.

Every simple object of $\cC_u$ is isomorphic to exactly one of the objects $E_\xi$ \textup{(}$\xi\in V^\vee$\textup{)}, $X_u(a,k)$ \textup{(}$a\neq0$, $k\in\Fp$\textup{)}. The isomorphism class of $X_u(a,k)$ does not depend on the choice of $r,s$, since it is the unique simple object of flux $a$ with $\rho(a)=g_a(a)\zeta^k$.

The characters $\chi_X(x):=\Tr\rho_X(x)$ are
\begin{equation*}
\chi_{E_\xi}(x)=\zeta^{\xi(x)},\qquad
\chi_{X_u(a,k)}(x)=\begin{cases}p\,g_a(x)\,\zeta^{kt}&\text{if }x=ta,\ t\in\Fp,\\ 0&\text{if }x\notin\Fp a.\end{cases}
\end{equation*}
These characters do not depend on $u$.
\end{proposition}

\begin{proof}
By Proposition~\ref{prop:centerabelian}, the simple objects of $\cC_u$ of flux $a$ are classified by the irreducible $\gamma_a^{\omega_u}$-projective representations of $V$. For $a=0$, one has $\gamma^{\omega_u}_0\equiv1$, so the irreducibles are precisely the characters $x\mapsto\zeta^{\xi(x)}$, $\xi\in V^\vee$ since $V$ is finite abelian. These give the $E_\xi$, and~\eqref{eq:halfexplicit} gives $E_0=\one$.

Now let $a\neq0$. A direct substitution in the defining formula for $\pi_{a,k}$ gives
\begin{equation*}
	\pi_{a,k}(x)\pi_{a,k}(y)=\zeta^{-\frac u2\det(a,x,y)}\pi_{a,k}(x+y)=b_a^u(x,y)\pi_{a,k}(x+y).
\end{equation*}
Hence Lemma~\ref{lem:multiplier} implies that $\rho=g_a\pi_{a,k}$ is a $\gamma_a^{\omega_u}$-projective representation. Moreover,
\begin{equation*}
	\pi_{a,k}(s)e_j=\zeta^{uj}e_j,\qquad \pi_{a,k}(r)e_j=e_{j+1}.
\end{equation*}
Since $u\neq0$, the first operator has $p$ distinct eigenspaces and the second permutes them transitively. Thus $\pi_{a,k}$, and hence $\rho$, is irreducible and they are pairwise non-isomorphic since $\pi_{a,k}(a)=\zeta^k\,\id$.

By Lemma~\ref{lem:maschke}, $\CC^{b_a^u}[V]$ is semisimple of dimension $p^3$, so the dimensions of its simple modules satisfy $\sum_U(\dim U)^2=p^3$~\cite[Proposition~3.5.8]{EGHLLSVY}. The $p$ irreducibles above already account for this dimension:
\begin{equation*}
	\sum_{k\in\Fp}(\dim\pi_{a,k})^2=p\cdot p^2=p^3.
\end{equation*}
Therefore they exhaust the irreducible $b_a^u$-projective representations.  Multiplication by the scalar cochain $g_a$ gives an equivalence between $b_a^u$-projective and $\gamma_a^{\omega_u}$-projective representations, hence the $\rho_{a,k}=g_a\pi_{a,k}$ exhaust the irreducible $\gamma_a^{\omega_u}$-projective representations. By Proposition~\ref{prop:centerabelian}, the $X_u(a,k)$ therefore exhaust the simple objects of flux $a$. At $x=a$ one has
\begin{equation*}
	\rho(a)=g_a(a)\zeta^k\,\id,
\end{equation*}
which distinguishes the $p$ classes at fixed flux and also shows that the isomorphism class of $X_u(a,k)$ is independent of the auxiliary choice of $r,s$.

Finally, the defining formula for $\pi_{a,k}$ gives
\begin{equation*}
	\Tr\pi_{a,k}(x)=\delta_{m_x,0}\,\zeta^{kt_x}\sum_{j\in\Fp}\zeta^{un_xj}=p\,\delta_{m_x,0}\delta_{n_x,0}\,\zeta^{kt_x},
\end{equation*}
since $u\neq0$. Thus
\begin{equation*}
	\chi_{X_u(a,k)}(x) = \begin{cases} 
	p\,g_a(x)\zeta^{kt},&x=ta,\ t\in\Fp,\\
	0,&x\notin\Fp a,
	\end{cases}
\end{equation*}
while $\chi_{E_\xi}(x)=\zeta^{\xi(x)}$.  Since $g_a$ is independent of $u$, so are these character formulas.
\end{proof}

\begin{remark}\label{rem:dims}
By Proposition~\ref{prop:centerabelian}, $d(E_\xi)=1, d(X_u(a,k))=p$. Hence Proposition~\ref{prop:simples} gives
\begin{equation*}
	\operatorname{rank}(\cC_u)=p^3+p(p^3-1)=p^4+p^3-p.
\end{equation*}
\end{remark}

\begin{definition}\label{def:Lambda}
For $u,v\in\Fp^\times$ let $\Lambda_{uv}:\Irr(\cC_u)\to\Irr(\cC_v)$ be $[E_\xi]\mapsto[E_\xi]$, $[X_u(a,k)]\mapsto[X_v(a,k)]$. By Proposition~\ref{prop:simples} it is a well-defined bijection, and $\Lambda_{uv}(\one)=\one$ and it preserves fluxes and characters.
\end{definition}

\subsection{Restriction of colors}

To compare link invariants under $\Lambda_{uv}$, we restrict the half-braidings to the subgroup generated by the fluxes. The restriction construction applies to any finite abelian group.

Let $\Gamma$ be a finite abelian group, $\nu$ a normalized $3$-cocycle, and $P\le\Gamma$ a subgroup. Let $\cZ(\Vect_\Gamma^\nu)_P$ be the full subcategory of objects $(X,c_{-,X})$ whose underlying object is supported on $P$.

\begin{proposition}\label{prop:restriction}
Let $\Gamma$ be a finite abelian group, let $\nu$ be a normalized $3$-cocycle, and let $P\leq\Gamma$, with $\nu_P:=\nu|_{P^3}$. The full subcategory $\cZ(\Vect_\Gamma^\nu)_P\subset\cZ(\Vect_\Gamma^\nu)$ consisting of objects whose underlying graded vector spaces are supported on $P$ is a fusion subcategory.  Restricting the half-braiding to $\Vect_P^{\nu_P}$ defines a $\CC$-linear ribbon functor
\begin{equation*}
	\Res_P:\cZ(\Vect_\Gamma^\nu)_P\longrightarrow\cZ(\Vect_P^{\nu_P}),\qquad(X,c_{-,X})\longmapsto\left(X,c_{-,X}|_{\Vect_P^{\nu_P}}\right),
\end{equation*}
where $X$ on the right is regarded as an object of $\Vect_P^{\nu_P}$, and it acts as the identity on morphisms in $\cZ(\Vect_\Gamma^\nu)_P$. Consequently, for every framed oriented link $L$ and colors $X_1,\ldots,X_n\in\cZ(\Vect_\Gamma^\nu)_P$,
\begin{equation*}
	\RT_{\cZ(\Vect_P^{\nu_P})}
   (L;\Res_PX_1,\ldots,\Res_PX_n)=\RT_{\cZ(\Vect_\Gamma^\nu)}
   (L;X_1,\ldots,X_n).
\end{equation*}
\end{proposition}

\begin{proof}
The tensor unit is supported at $0\in P$, and support in $P$ is preserved by finite direct sums, direct summands, tensor products and duals. Hence $\cZ(\Vect_\Gamma^\nu)_P$ is a fusion subcategory.

If $(X,c_{-,X})$ is supported on $P$, then $X$ is canonically an object of $\Vect_P^{\nu_P}$, and restricting the family $c_{-,X}$ to objects of $\Vect_P^{\nu_P}$ preserves the half-braiding identity. Likewise, every morphism in $\cZ(\Vect_\Gamma^\nu)_P$ remains a morphism after restriction. Thus $\Res_P$ is a well-defined $\CC$-linear functor.

Since the monoidal and rigid structures on $\Vect_P^{\nu_P}$ are the restrictions of those on $\Vect_\Gamma^\nu$, $\Res_P$ is strict monoidal and preserves the dual objects, evaluation maps and coevaluation maps. Restricting half-braidings also preserves the braiding. The positive spherical structure on $\Vect_\Gamma^\nu$ restricts to that on $\Vect_P^{\nu_P}$, and hence the induced spherical structures on their centers are preserved. Since the Drinfeld isomorphism $\upsilon_X:X\to X^{**}$ is constructed functorially from the braiding, evaluation and coevaluation maps, the preceding compatibilities imply
\begin{equation*}
	\Res_P(\upsilon_X)=\upsilon_{\Res_PX}.
\end{equation*}
The induced spherical structures are also preserved, so
\begin{equation*}
	\Res_P(\psi_X)=\psi_{\Res_PX}.
\end{equation*}
Hence, using $\theta=\upsilon^{-1}\psi$,
\begin{equation*}
	\Res_P(\theta_X)=\upsilon_{\Res_PX}^{-1}\psi_{\Res_PX}=\theta_{\Res_PX}.
\end{equation*}
Thus $\Res_P$ is braided monoidal and preserves the canonical twist, so it is a ribbon functor. The asserted equality of Reshetikhin--Turaev evaluations follows from the functoriality recalled in Section~\ref{subsec:RT-conventions}.
\end{proof}

Let $u,v\in\Fp^\times$ and let $P\le V$ with $\dim P\le2$. By Lemma~\ref{lem:rank2}, the functors $\Res_P$ from $(\cC_u)_P$ and from $(\cC_v)_P$ have the same target
\begin{equation*}
\cB_P:=\cZ(\Vect_P^{\kappa|_{P^3}}),
\end{equation*}
with the same ribbon structure.

\begin{lemma}\label{lem:match}
Let $X\in\Irr(\cC_u)$ have flux in $P$. Then $\Res_PX\cong\Res_P(\Lambda_{uv}X)$ in $\cB_P$.
\end{lemma}

\begin{proof}
For $X=E_\xi$, the two restrictions are identical, namely both are $\CC$ in degree zero with action $\rho(x)=\zeta^{\xi(x)}, x\in P$.

Now let $X=X_u(a,k)$ with $a\in P\setminus\{0\}$, and put $X':=X_v(a,k)$. By Lemma~\ref{lem:rank2}, the restricted projective representations $\rho|_P$ and $\rho'|_P$ have the same multiplier $\partial g_a|_{P\times P}$. Hence
\begin{equation*}
	\sigma(x):=g_a(x)^{-1}\rho(x),\qquad\sigma'(x):=g_a(x)^{-1}\rho'(x)
\end{equation*}
are ordinary representations of $P$. By Proposition~\ref{prop:simples}, the characters of $X_u(a,k)$ and $X_v(a,k)$ coincide. Since $\sigma=g_a^{-1}\rho$ and $\sigma'=g_a^{-1}\rho'$, their ordinary characters also coincide:
\begin{equation*}
	\Tr\sigma(x)=g_a(x)^{-1}\chi_X(x)=g_a(x)^{-1}\chi_{X'}(x)=\Tr\sigma'(x).
\end{equation*}
Therefore $\sigma\cong\sigma'$ by the usual character criterion for finite-group representations. Multiplying back by $g_a$ gives $\rho|_P\cong\rho'|_P$. By Proposition~\ref{prop:centerabelian}, this is precisely an isomorphism
\begin{equation*}
	\Res_PX\cong\Res_PX'
\end{equation*}
in $\cB_P$.
\end{proof}

For a link with at most two components, the fluxes of its colors span a subspace $P\le V$ of dimension at most two. Hence Lemma~\ref{lem:match} compares all of its colors inside the common ribbon category $\cB_P$.

\begin{theorem}[Equality of link invariants]\label{thm:step1}
Let $u,v\in\Fp^\times$. For every framed oriented link $L$ with $n\in\{1,2\}$ components and all $X_1,\dots,X_n\in\Irr(\cC_u)$,
\begin{equation*}
\RT_{\cC_u}(L;X_1,\dots,X_n)=\RT_{\cC_v}(L;\Lambda_{uv}X_1,\dots,\Lambda_{uv}X_n).
\end{equation*}
Moreover $\theta_{\Lambda_{uv}X}=\theta_X$ and $d(\Lambda_{uv}X)=d(X)$ for every simple $X$.
\end{theorem}

\begin{proof}
Let $a_j$ be the flux of $X_j$ and $P:=\spanop(a_1,\dots,a_n)$, so $\dim P\le n\le2$. The objects $X_j$ lie in $(\cC_u)_P$, and $\Lambda_{uv}X_j$ with the same flux lie in $(\cC_v)_P$. By Proposition~\ref{prop:restriction}, Lemma~\ref{lem:match} and Section~\ref{subsec:RT-conventions} restricted to $\cB_P$),
\begin{align*}
\RT_{\cC_u}(L;X_1,\dots,X_n)&=\RT_{\cB_P}(L;\Res_PX_1,\dots,\Res_PX_n)\\
&=\RT_{\cB_P}(L;\Res_P\Lambda_{uv}X_1,\dots,\Res_P\Lambda_{uv}X_n)=\RT_{\cC_v}(L;\Lambda_{uv}X_1,\dots,\Lambda_{uv}X_n).
\end{align*}
To compare twists, take $P=\Fp a$, where $a$ is the flux of $X$. By Proposition~\ref{prop:restriction}, $\theta_X$ and $\theta_{\Lambda_{uv}X}$ are the twists of the isomorphic objects $\Res_PX\cong\Res_P\Lambda_{uv}X$ of $\cB_P$, and the twist is natural, so they are equal. The equality $d(\Lambda_{uv}X)=d(X)$ comes from the case $L=O^0$.
\end{proof}

\begin{corollary}\label{cor:STW}
Under the single bijection $\Lambda_{uv}$, the matrices $\widetilde S$, $S$, $T$, $\widetilde W$ and $W$ of $\cC_u$ and $\cC_v$ coincide. This proves part~\textup{(A)} of the Main Theorem.
\end{corollary}

\begin{proof}
This follows from Theorem~\ref{thm:step1} and Section~\ref{subsec:STW-conventions}.
\end{proof}

\begin{remark}\label{rem:uzero}
Although Lemma~\ref{lem:rank2} also holds for $u=0$, the comparison with $\cC_u$, $u\neq0$, cannot hold under a bijection of simple objects. Indeed, for $u=0$ every multiplier $\gamma_a^{\omega_0}=\partial g_a$ is a coboundary, so each flux has $p^3$ simple objects and $\operatorname{rank}(\cC_0)=p^6$. For $u\neq0$, Remark~\ref{rem:dims} gives $\operatorname{rank}(\cC_u)=p^4+p^3-p$. Thus the assumption $u,v\neq0$ is essential for the common $\Lambda_{uv}$ used in Theorem~\ref{thm:step1}.
\end{remark}

%======================================================================
\section{The obstruction to braided equivalence}\label{sec:step2}
%======================================================================

Throughout this section $u,v\in\Fp^\times$, and $\Phi:\cC_u\to\cC_v$ denotes a braided tensor equivalence with tensor structure $J_{X,Y}:\Phi(X)\otimes\Phi(Y)\to\Phi(X\otimes Y)$ where no compatibility with ribbon twists $\theta_X$ is assumed.

\subsection{The electric subcategory and its regular algebra}

The invertible objects identify the electric subcategory intrinsically. This forces a braided equivalence to preserve the regular algebra that will be used to recover the pointed category.

\begin{lemma}\label{lem:invertible}
A simple object of $\cC_u$ is invertible if and only if it is isomorphic to some $E_\xi$, and $E_\xi\otimes E_{\xi'}\cong E_{\xi+\xi'}$.

Let $\cE_u\subseteq\cC_u$ be the full subcategory of objects supported in degree $0$. Its objects are exactly the direct sums of invertible simple objects of $\cC_u$. It is a fusion subcategory, and $(X,\rho_X)\mapsto(X,\rho_X)$ is an isomorphism of braided monoidal categories $\cE_u\to\Rep(V)$, where $\Rep(V)$ has the trivial associativity constraint and the flip braiding.
\end{lemma}

\begin{proof}
By Proposition~\ref{prop:centerabelian}, a simple object is invertible exactly when its underlying space is one-dimensional. Proposition~\ref{prop:simples} gives dimension $1$ for the $E_\xi$ and dimension $p>1$ for the $X_u(a,k)$. Proposition~\ref{prop:centerabelian} gives $\rho_{E_\xi\otimes E_{\xi'}}(x)=\zeta^{\xi(x)+\xi'(x)}$, which proves the tensor product formula. That $\cE_u\to\Rep(V)$ is an isomorphism follows from~\cite[\S1.1, immediately after Theorem~1.2]{NN}.
\end{proof}

It follows that the electric subcategory is preserved by $\Phi$.

\begin{lemma}\label{lem:Epreserved}
$\Phi(\cE_u)\subseteq\cE_v$, and $\Phi$ restricts to a braided tensor equivalence $\cE_u\to\cE_v$.
\end{lemma}

\begin{proof}
An equivalence sends simple objects to simple objects, and the tensor structure sends invertible objects to invertible objects because $\Phi(X)\otimes\Phi(Y)\cong\Phi(X\otimes Y)\cong\Phi(\one)\cong\one$. Lemma~\ref{lem:invertible} therefore implies that $\Phi(E_\xi)$ is isomorphic to some $E_{\xi'}$. Since $\cE_u$ consists of direct sums of the $E_\xi$, and $\cE_v$ is closed under direct sums and isomorphisms, it follows that $\Phi(\cE_u)\subseteq\cE_v$. 

Let $\Psi:\cC_v\to\cC_u$ be a monoidal tensor quasi-inverse of $\Phi$~\cite[Remark~2.4.10]{EGNO} with monoidal isomorphisms $\varepsilon:\Phi\Psi\Rightarrow\id$ and $\Psi\Phi\cong\id$, and we intend to show that it is also braided. For $X,Y\in\mathcal C_v$ put $f=\Psi(c_{X,Y})\,J^\Psi_{X,Y}$ and $g=J^\Psi_{Y,X}\,c_{\Psi X,\Psi Y}$. Monoidality of $\varepsilon$ means $\Phi(J^\Psi_{X,Y})J^\Phi_{\Psi X,\Psi Y}=\varepsilon_{X\otimes Y}^{-1}(\varepsilon_X\otimes\varepsilon_Y)$. Using this together with naturality of $\varepsilon$ and of $c$, and that $\Phi$ is braided: $\Phi(f)\,J^\Phi=\varepsilon^{-1}_{Y\otimes X}\,c_{X,Y}\,(\varepsilon_X\otimes\varepsilon_Y)=\Phi(g)\,J^\Phi$. Since $J^\Phi$ is invertible and $\Phi$ is faithful, $f=g$. So $\Psi$ is braided.

The same argument gives $\Psi(\cE_v)\subseteq\cE_u$. Hence, for every $Y\in\cE_v$, $Y\cong\Phi\Psi(Y)$ with $\Psi(Y)\in\cE_u$, so $\Phi|_{\cE_u}:\cE_u\to\cE_v$ is essentially surjective. It is fully faithful because $\Phi$ is fully faithful and $\cE_u,\cE_v$ are full subcategories of $\cC_u,\cC_v$, respectively.  Thus $\Phi|_{\cE_u}$ is a braided tensor equivalence.
\end{proof}

For a finite group $G$, the regular algebra in $\Rep(G)$ is $\Fun(G)$ with pointwise multiplication and the action of $G$ by left translations. It is commutative and separable, with a bimodule splitting of multiplication given in \cite[Example~2.8]{DMNO}. An algebra is called \'etale if it is commutative and separable, and connected if $\dim\Hom(\one,A)=1$ \cite[Definition~3.1]{DMNO}.

Let $A_u\in\cE_u$ be the algebra corresponding to $\Fun(V)$ under the isomorphism of Lemma~\ref{lem:invertible}.

\begin{lemma}\label{lem:regular}
$A_u$ is a connected \'etale algebra in $\cC_u$, and $A_u\cong\bigoplus_{\xi\in V^\vee}E_\xi$ as an object. Moreover, every connected \'etale algebra $B$ in $\cE_v$ whose underlying vector space has dimension $|V|$ is isomorphic to $A_v$ as an algebra. Consequently, $\Phi(A_u)\cong A_v$ as algebras in $\cC_v$.
\end{lemma}

\begin{proof}
Since $\mathcal E_u\subset\mathcal C_u$ is Tannakian, its regular algebra $A_u$ is connected \'etale by~\cite[Example~3.3(i)]{DMNO}. The decomposition $A_u\cong\bigoplus_{\xi\in V^\vee}E_\xi$ follows from the regular-representation decomposition~\cite[Theorem~4.1.1(ii)]{EGHLLSVY}, since $V$ is abelian.

Under the braided equivalence $\cE_v\simeq\Rep(V)$, connected \'etale algebras correspond to function algebras on transitive finite $V$-sets~\cite[Lemma~3.3.1]{Davydov}. Thus $B\cong\Fun(V/H)$ for some subgroup $H\leq V$. Since $\dim_{\CC}B=[V:H]=|V|$,  we have $H=\{0\}$, and hence $B\cong\Fun(V)=A_v$ as algebras. 

Since $\Phi$ is a braided tensor equivalence, $\Phi(A_u)$ is a connected \'etale algebra in $\cE_v$ of dimension $|V|$. By the preceding classification of connected \'etale algebras in $\cE_v\simeq\Rep(V)$, it follows that $\Phi(A_u)\cong A_v$ as algebras.
\end{proof}

\subsection{Recovering the pointed category}

For a connected \'etale algebra $A$ in a braided fusion category $\cC$, write $\cC_A$ for the category of right $A$-modules $(M,\mu)$, where $\mu:M\otimes A\longrightarrow M$ is the right action, equipped with the tensor structure of~\cite[\S3.3]{DMNO}. Namely, a right $A$-module $M$ is regarded as an $A$-bimodule $M_-$ using the left action $A\otimes M\xrightarrow{c^{-1}_{M,A}}M\otimes A\longrightarrow M$~\cite[\S3.3]{DMNO}, and the tensor product in $\cC_A$ is induced by the relative tensor product of these bimodules.

\begin{proposition}\label{prop:reconstruction}
For $u\in\Fp^\times$, $(\cC_u)_{A_u}$ is monoidally equivalent to $\Vect_V^{\omega_u}$.
\end{proposition}

\begin{proof}
Let $I_u$ be the right adjoint of  $\Forg:\cC_u\to\Vect_V^{\omega_u}$. By \cite[Example~8.8.9(ii)]{EGNO}, $I_u(\one)$ is the regular algebra of the canonical electric subcategory $\Rep(V)\subset\cC_u$, hence is isomorphic to $A_u$ as an algebra. Applying \cite[Lemma~8.12.2(ii)]{EGNO} gives the monoidal equivalence
\begin{equation*}
    (\cC_u)_{A_u}\simeq(\cC_u)_{I_u(\one)}
\simeq\Vect_V^{\omega_u}.
\end{equation*}
\end{proof}

\begin{lemma}\label{lem:transport}
Let $\Psi:\cC\to\cD$ be a braided tensor equivalence between braided fusion categories, with tensor structure $J$, and let $A$ be a connected \'etale algebra in $\cC$. Then $(M,\mu)\mapsto(\Psi M,\Psi(\mu)\circ J_{M,A})$ is a monoidal equivalence $\cC_A\to\cD_{\Psi(A)}$. If $\varphi:B\to B'$ is an isomorphism of connected \'etale algebras in $\cD$, then $(M,\mu)\mapsto
(M,\mu\circ(\id_M\otimes\varphi^{-1}))$ defines a monoidal equivalence $\cD_B\to\cD_{B'}$.
\end{lemma}

\begin{proof}
A strong monoidal functor transports algebra objects and their modules. For the tensor structure on $\cC_A$ used here, recall that a right $A$-module $M$ is regarded as the bimodule $M_-$ via~\cite[\S3.3]{DMNO}. Since $\Psi$ is braided, it carries this left action to the corresponding left action on $(\Psi M)_-$. Moreover, as an equivalence of abelian categories, $\Psi$ preserves the coequalizers defining relative tensor products. Hence there are natural isomorphisms $\Psi M\otimes_{\Psi A}\Psi N\cong\Psi(M\otimes_A N)$ compatible with associativity and units. Thus $\cC_A\simeq_\otimes\cD_{\Psi(A)}$. A braided tensor quasi-inverse of $\Psi$ gives the inverse equivalence.

The second assertion is obtained by transporting module structures along the algebra isomorphism $\varphi:B\to B'$.
\end{proof}

\begin{corollary}\label{cor:pointedequiv}
If $\Phi:\cC_u\to\cC_v$ is a braided tensor equivalence, then $\Vect_V^{\omega_u}$ and $\Vect_V^{\omega_v}$ are monoidally equivalent.
\end{corollary}

\begin{proof}
Proposition~\ref{prop:reconstruction} recovers the two pointed categories from their regular algebras. Lemma~\ref{lem:transport} carries the module category along $\Phi$, and Lemma~\ref{lem:regular} identifies the transported algebra with $A_v$. Applying Lemma~\ref{lem:transport} to that algebra isomorphism gives
\begin{equation*}
\Vect_V^{\omega_u}\simeq(\cC_u)_{A_u}\simeq(\cC_v)_{\Phi(A_u)}
\cong(\cC_v)_{A_v}\simeq\Vect_V^{\omega_v}.
\end{equation*}
\end{proof}

\begin{remark}
Corollary~\ref{cor:pointedequiv} can also be obtained from the classification of Lagrangian subcategories in~\cite[Theorem~4.5]{DGNO07}. Our approach above gives a direct reconstruction through the regular algebra $A_u$.
\end{remark}

\subsection{The cocycle class}

A monoidal equivalence of pointed categories determines an automorphism of the grading group and a relation between the associators. We spell out that relation with our conventions.

For $\beta:\Gamma\times\Gamma\to\CC^\times$ we put $(\partial\beta)(x,y,z):=\beta(y,z)\,\beta(x,y+z)\,\beta(x+y,z)^{-1}\beta(x,y)^{-1}$. For a $3$-cochain $\nu:\Gamma^3\to\CC^\times$ and $M\in\Aut(\Gamma)$, we put $(M^*\nu)(x,y,z):=\nu(Mx,My,Mz)$.

\begin{lemma}\label{lem:pointed}
Let $\nu,\nu'$ be normalized $3$-cocycles on a finite abelian group $\Gamma$. Then $\Vect_\Gamma^\nu\simeq_\otimes\Vect_\Gamma^{\nu'}$ if and only if there exist $M\in\Aut(\Gamma)$ and a $2$-cochain $\beta$ such that $M^*\nu'=\nu\,\partial\beta$.
\end{lemma}
This is a standard result by~\cite[Equation~(2.31)]{EGNO}.

To extract the restrictions imposed by this relation, we use two functions of the cocycle class. For 3-cocycle $\nu:V^3\to\CC^\times$ we define~\cite[Equation~(3.14)]{AC}~\cite[Lemma~2.12]{DavydovSimmons}
\begin{equation*}
    q_\nu(a):=\prod_{t\in\Fp}\nu(a,ta,a),\qquad\Alt(\nu)(a_1,a_2,a_3):=\prod_{\sigma\in S_3}\nu\bigl(a_{\sigma(1)},a_{\sigma(2)},a_{\sigma(3)}\bigr)^{\sgn(\sigma)}.
\end{equation*}

\begin{lemma}\label{lem:invariants}
The functions $q$ and $\Alt$ are multiplicative, with $q_{\nu\nu'}=q_\nu q_{\nu'}$ and $\Alt(\nu\nu')=\Alt(\nu)\Alt(\nu')$. They are trivial on coboundaries, so $q_{\partial\beta}\equiv1$ and $\Alt(\partial\beta)\equiv1$ for every $\beta:V\times V\to\CC^\times$. They are also natural under $M\in GL_3(\Fp)$, in the sense that $q_{M^*\nu}(a)=q_\nu(Ma)$ and $\Alt(M^*\nu)(a,b,c)=\Alt(\nu)(Ma,Mb,Mc)$.

For the cocycles $\omega_u$, these functions are $q_{\omega_u}(a)=\zeta^{Q(a)}$ and $\Alt(\omega_u)(a,b,c)=\zeta^{u\det(a,b,c)}$ for every $u\in\Fp$.
\end{lemma}

\begin{proof}
Multiplicativity follows directly from the definitions, and naturality follows from $M(ta)=tMa$.

For the triviality of $q_\nu$ on a coboundary, expand $(\partial\beta)(a,ta,a)=\beta(ta,a)\beta(a,(t+1)a)\,\beta((t+1)a,a)^{-1}\beta(a,ta)^{-1}$. Since translation by $1$ permutes $\Fp$, telescoping gives $\prod_t\beta(ta,a)/\beta((t+1)a,a)=1$ and $\prod_t\beta(a,(t+1)a)/\beta(a,ta)=1$. Thus $q_{\partial\beta}\equiv1$.

For $\Alt_\nu$, the expression of $\Alt(\nu)(a_1,a_2,a_3)$ is the cocycle-level representative of the standard alternation map $\Alt_3:H^3(V,\CC^\times)\longrightarrow\Hom(\Lambda^3V,\CC^\times)$~\cite[Lemma~2.12]{DavydovSimmons}. In particular, it depends only on the cohomology class of the cocycle. Thus $\Alt(\partial\beta)\equiv1$ since every coboundary is trivial.

To evaluate $q$ and $\Alt$ on $\omega_u=\kappa\tau_u$, first note that $\det(a,ta,a)=0$, so $q_{\tau_u}(a)=1$.  Moreover, for $r\in\Fp$, $\sum_{t\in\Fp}C_p(tr,r)=[r]$ since $t\mapsto t+1$ permutes $\Fp$. Hence $q_{\kappa}(a)=\zeta^{\sum_i[a_i]^2}=\zeta^{Q(a)}$, and therefore $q_{\omega_u}(a)=\zeta^{Q(a)}$.

For $\Alt$, symmetry of $C_p$ gives $\Alt(\kappa)=1$. Since $\det$ is alternating,
\begin{equation*}
	\Alt(\tau_u)(a_1,a_2,a_3)=\prod_{\sigma\in S_3}\zeta^{\frac u6\sgn(\sigma)\det(a_{\sigma(1)},a_{\sigma(2)},a_{\sigma(3)})}=\zeta^{u\det(a_1,a_2,a_3)}.
\end{equation*}
Thus, by multiplicativity, $\Alt(\omega_u)(a,b,c)=\zeta^{u\det(a,b,c)}$.
\end{proof}

The cyclic invariant and the alternation now give the obstruction to braided equivalence.

\begin{theorem}[Obstruction to braided equivalence]\label{thm:step2}
Let $u,v\in\Fp^\times$. If $\cC_u$ and $\cC_v$ are braided equivalent, then $u=\pm v$.
\end{theorem}

\begin{proof}
Assume that $\Phi:\cC_u\longrightarrow\cC_v$ is a braided tensor equivalence. By Corollary~\ref{cor:pointedequiv}, $\Phi$ induces a monoidal equivalence $\Vect_V^{\omega_u}\simeq_\otimes\Vect_V^{\omega_v}$. Applying Lemma~\ref{lem:pointed} to this monoidal equivalence gives $M\in GL_3(\Fp)$ and a $2$-cochain $\beta:V\times V\to\CC^\times$ such that $M^*\omega_v=\omega_u\,\partial\beta$.

We now apply the two cohomological invariants of Lemma~\ref{lem:invariants} to this relation. Applying $q$ to $M^*\omega_v=\omega_u\,\partial\beta$ gives
\begin{equation*}
	\zeta^{Q(Ma)}=q_{\omega_v}(Ma)=q_{M^*\omega_v}(a)=q_{\omega_u}(a)q_{\partial\beta}(a)=\zeta^{Q(a)}
\end{equation*}
Hence $Q(Ma)=Q(a)$ for all $a\in V$. Polarizing $Q(a)$ gives $B(x,y):=Q(x+y)-Q(x)-Q(y)=2x^{\mathsf T}y$, so $B(Mx,My)=B(x,y)$.  Since $2\in\Fp^\times$, this implies $M^{\mathsf T}M=I$, therefore $\det M=\pm1$.

Applying $\Alt$ to $M^*\omega_v=\omega_u\,\partial\beta$ gives
\begin{equation*}
	\zeta^{v\det(Ma,Mb,Mc)}=\Alt(\omega_v)(Ma,Mb,Mc)=\Alt(M^*\omega_v)(a,b,c)=\Alt(\omega_u)(a,b,c)=\zeta^{u\det(a,b,c)}
\end{equation*}
where $\Alt(\partial\beta)=1$ (Lemma~\ref{lem:invariants}) has been used. Evaluating at $(a,b,c)=(\mathbf e_1,\mathbf e_2,\mathbf e_3)$ yields $v\det M=u$. Since $\det M=\pm1$, we conclude $u=\pm v$.
\end{proof}

\begin{remark}\label{rem:converse}
Although it is not needed for the obstruction, the converse holds. Let $I:V\to V, I(x)=-x$ and define $\epsilon(r)=0$ for $r=0$ and $\epsilon(r)=1$ otherwise, and $\beta_I(a,b):=\zeta^{\sum_i[a_i]\epsilon(b_i)}$. Since $[-r]=-[r]+p\epsilon(r)$, the definition of the carry gives
\begin{equation*}
	C_p(-b,-c)=-C_p(b,c)+\epsilon(b)+\epsilon(c)-\epsilon(b+c).
\end{equation*}
A direct substitution then yields $I^*\kappa=\kappa\,\partial\beta_I$. On the other hand, since the determinant is trilinear, $I^*\tau_u=\tau_{-u}$. Hence, $I^*\omega_u=\omega_{-u}\,\partial\beta_I$.

Reading Lemma~\ref{lem:pointed} in the converse direction, this relation defines a monoidal equivalence $\Vect_V^{\omega_{-u}}\simeq_\otimes\Vect_V^{\omega_u}$ with grading automorphism $I$. Explicitly, $F(X)_y=X_{-y}$ and the tensorator acts by $\beta_I(x,y)^{-1}$ on $X_x\otimes Y_y$. Passing to Drinfel'd centers therefore gives $\cC_{-u}\simeq_{\mathrm{br}}\cC_u$. Together with Theorem~\ref{thm:step2}, it follows that the categories $\cC_u$, $u\in\Fp^\times$, form exactly $(p-1)/2$ braided equivalence classes.
\end{remark}

%======================================================================
\section{Completion of the proof}\label{sec:final}
%======================================================================

\begin{proof}[Proof of the Main Theorem]
Part~(A) follows from Theorem~\ref{thm:step1} and Corollary~\ref{cor:STW}. Part~(B) is Theorem~\ref{thm:step2}, since every ribbon equivalence is a braided equivalence~\cite[\S2.4]{KT}.
\end{proof}

\begin{proof}[Proof of Corollary~\ref{cor:p5}]
In $\mathbb F_5$ one has $-1=4$, so $2\notin\{1,4\}=\{\pm1\}$. By (A), $\cC_1$ and $\cC_2$ have the same $\widetilde S,S,T,\widetilde W,W$ under the vacuum-preserving labeling $\Lambda_{12}$. Part~(B) rules out a braided equivalence and therefore a ribbon equivalence. The rank and dimensions are given in Remark~\ref{rem:dims}.
\end{proof}

\begin{remark}
Within the proof of the Main Theorem, the ribbon structure is used only in Proposition~\ref{prop:restriction}, where the restriction functor is shown to preserve the canonical twist, and in Theorem~\ref{thm:step1}, where ribbon functoriality of the Reshetikhin--Turaev construction gives equality of framed link invariants. Corollary~\ref{cor:STW} then identifies the resulting $S$, $T$, and $W$ data.

The obstruction in Theorem~\ref{thm:step2} uses only the braided tensor structure. Lemma~\ref{lem:invertible} identifies the electric subcategory intrinsically, the results on connected \'etale algebras and module categories in \cite{DMNO} reconstruct the underlying pointed category, and Lemmas~\ref{lem:pointed} and~\ref{lem:invariants} then give the cocycle obstruction. Hence part~(B) excludes every braided equivalence, not only the ones inducing $\Lambda_{uv}$.

\end{remark}

\begin{remark}
The braided-equivalence classification in Theorem~\ref{thm:step2} and Remark~\ref{rem:converse} also follows from the classification of pointed fusion categories of dimension $p^3$ up to weak Morita equivalence in~\cite[\S1.1 and \S2.7]{MMU}, together with~\cite[Theorem~8.12.3]{EGNO}. In the cohomological coordinates of \cite{MMU}, our cocycle class is $[\omega_u]=y_1^2+y_2^2+y_3^2+u\,\beta(x_1x_2x_3)$. In this work we do not provide a proof of Theorem~\ref{thm:step2} along this approach.
\end{remark}

%======================================================================
\section{Uncolored 3-manifold partition function}
\label{sec:observable}
%======================================================================

The preceding proof gives inequivalent categories whose colored RT invariants agree for every link with at most two components. We now distinguish these classes by the partition function on a closed oriented three-manifold, with no line insertions or chosen anyon labels. The calculation combines a three-component Borromean amplitude with the common twists through surgery.

Let $B$ be the zero-framed Borromean link obtained by closing the braid
\begin{equation*}
	\mathfrak{b}_{\mathrm{Bor}}=(\sigma_2^{-1}\sigma_1)^3,
\end{equation*}
with all strands oriented in the braid direction, positive generators represented by the braiding, and the rightmost factor acting first on a left-associated tensor product. Here $\sigma_i$ is the positive crossing between strands currently in positions $i$ and $i+1$, and $\sigma_i^{-1}$ is the inverse crossing. Thus the word repeats $\sigma_1$ followed by $\sigma_2^{-1}$ three times, giving six crossings in total. This is the braid used to define the Borromean tensor in~\cite[Definition~4.1]{KMS}. Let $M_p$ be the oriented manifold obtained by integral surgery on $B$ with coefficient $+p$ on each component, measured relative to the zero framing. Its linking matrix is $pI_3$, so $H_1(M_p,\mathbb Z)\cong(\mathbb Z/p)^3$. Write $Z_u(M)$ for the RT invariant of $M$ associated with $\cC_u$, normalized by $Z_u(S^3)=D^{-1}=p^{-3}$~\cite[Chapter~II, \S2.2, equation~(2.2.b)]{Tu}.

\begin{theorem}\label{thm:observable}
For every prime $p\ge5$ and $u\in\Fp^\times$,
\begin{equation}\label{eq:observable-value}
	Z_u(M_p)=p\,G_p(-1)+p(p^2-1)\,\zeta^{-4u^{-2}},\qquad G_p(-1):=\sum_{t\in\Fp}\zeta^{-t^2}.
\end{equation}
Consequently,
\begin{equation}\label{eq:observable-separates}
	Z_u(M_p)=Z_v(M_p)\quad\Longleftrightarrow\quad u=\pm v.
\end{equation}
Thus, for each fixed $p$, this single uncolored partition function separates all $(p-1)/2$ braided equivalence classes of the family.
\end{theorem}

\subsection{Twists and the Borromean amplitude}

First we fix the signs entering the surgery calculation. For a simple object $X$ of flux $a$, Proposition~\ref{prop:simples} makes $\rho_X(a)$ a scalar $\lambda_X$, and~\eqref{eq:halfexplicit} gives $c_{X,X}=\lambda_X\,\mathrm{flip}$. The identity $\theta_Xd(X)=\Tr(c_{X,X})$ of~\cite[Remark~2.33]{DGNO10}, together with the positive pivotal structure, gives $\theta_X=\lambda_X$, since the pivotal trace here is the ordinary linear trace and the trace of the flip on $X\otimes X$ is $\dim X$. Hence
\begin{equation}\label{eq:observable-twists}
\theta_{E_\xi}=1,
\qquad
\theta_{X_u(a,k)}=g_a(a)\zeta^k
=\eta^{-\sum_i[a_i]^2}\zeta^k,
\qquad
\theta_{X_u(a,k)}^p=\zeta^{-Q(a)}.
\end{equation}
In particular, the sign of the last exponent is negative with the conventions of Section~\ref{sec:prelim}.

For simple objects $X,Y,Z$ with fluxes $a,b,c$, respectively, write $B_u(X,Y,Z):=\RT_{\cC_u}(B;X,Y,Z)$. The six crossings act in the order $\sigma_1,\sigma_2^{-1},\sigma_1,\sigma_2^{-1},\sigma_1,\sigma_2^{-1}$, during which the three objects return to their original order and the operators $\rho_X$ on each object cancel in inverse pairs. The three pairs of reassociations contribute the scalar
\begin{equation*}
	\frac{\omega_u(b,a,c)}{\omega_u(b,c,a)}\frac{\omega_u(c,b,a)}{\omega_u(c,a,b)}\frac{\omega_u(a,c,b)}{\omega_u(a,b,c)}=\Alt(\omega_u)(a,b,c)^{-1}.
\end{equation*}
Taking the pivotal trace and applying Lemma~\ref{lem:invariants} gives
\begin{equation}\label{eq:observable-borromean}
	B_u(X,Y,Z)=d(X)d(Y)d(Z)\,\zeta^{-u\det(a,b,c)}.
\end{equation}
This includes zero fluxes and is independent of the projective-character labels. The braid closure has zero self-framing on each component, since every crossing is between distinct components. Its pairwise linking numbers vanish, although its three-component amplitude detects the alternating cocycle.

\subsection{The surgery sum}

The dimension formulas give $D^2=\sum_Xd(X)^2=p^6$ and, for every flux $a$,
\begin{equation}\label{eq:observable-flux-weight}
	\sum_{\substack{X\in\Irr(\cC_u)\\\mathrm{flux}(X)=a}}d(X)^2=p^3.
\end{equation}
Moreover,~\eqref{eq:observable-twists} implies
\begin{equation*}
	\Delta_\pm:=\sum_X d(X)^2\theta_X^{\pm1}=p^3+p^2\sum_{a\ne0}g_a(a)^{\pm1}\sum_{k\in\Fp}\zeta^{\pm k}=p^3=D.
\end{equation*}
Thus the signature-dependent factor in the RT surgery formula is trivial. In our normalization, if a framed $m$-component link $J$ presents $M$, then
\begin{equation}\label{eq:observable-surgery}
Z_u(M)=D^{-m-1}
\sum_{X_1,\ldots,X_m\in\Irr(\cC_u)}
\left(\prod_{j=1}^m d(X_j)\right)
\RT_{\cC_u}(J;X_1,\ldots,X_m),
\end{equation}
by~\cite[Chapter~II, \S2.2, equation~(2.2.a) and Theorem~2.2.2]{Tu}.

Adding $p$ positive framing twists to each component of $B$ multiplies~\eqref{eq:observable-borromean} by $\theta_X^p\theta_Y^p\theta_Z^p$. Combining~\eqref{eq:observable-twists}--\eqref{eq:observable-surgery} and summing first over the colors at each fixed flux yields
\begin{equation}\label{eq:observable-finite-sum}
Z_u(M_p)=\frac1{p^3}
\sum_{a,b,c\in\Fp^3}
\zeta^{-Q(a)-Q(b)-Q(c)-u\det(a,b,c)}.
\end{equation}
Indeed, each dimension weight in the surgery sum multiplies one dimension factor from the Borromean amplitude, and the resulting three sums of squared dimensions contribute $(p^3)^3$, hence the overall factor is $D^{-4}(p^3)^3=p^{-3}$.

\subsection{Evaluation of the finite sum}

\begin{proof}[Proof of Theorem~\ref{thm:observable}]
Set $g=G_p(-1)$, $s=\left(\frac{-1}{p}\right)$ where $\left(\frac{\cdot}{p}\right)$ is the Legendre symbol, and $t=u^2/4\in\Fp^\times$. We use the quadratic Gauss-sum identities $g^2=sp$ and $\sum_{z\in\Fp}\zeta^{h z^2}=\left(\frac{-h}{p}\right)g$ for $h\ne0$. All dot products, cross products and matrices below are over $\Fp$.

Since $\det(a,b,c)=(a\times b)\cdot c$, completing the square in $c$ gives
\[
\sum_{c\in\Fp^3}\zeta^{-Q(c)-u(a\times b)\cdot c}
=g^3\zeta^{tQ(a\times b)}.
\]
Using $Q(a\times b)=Q(a)Q(b)-(a\cdot b)^2$, the remaining sum in $b$ is the quadratic Gauss sum for
\[
K_a=-I+t\bigl(Q(a)I-aa^{\mathsf T}\bigr),
\qquad
\det K_a=-\bigl(1-tQ(a)\bigr)^2.
\]
If $Q(a)\ne t^{-1}$, this matrix is nondegenerate and its determinant has quadratic character $s$, so
\[
\sum_{b\in\Fp^3}\zeta^{b^{\mathsf T}K_ab}=g^3.
\]
The determinant argument includes nonzero isotropic $a$, for which $Q(a)=0$.

If $Q(a)=r:=t^{-1}=4u^{-2}$, then $a\ne0$ and $K_a=-t aa^{\mathsf T}$ has rank one. The linear form $b\mapsto a\cdot b$ is surjective with fibers of size $p^2$, and $t$ is a square, hence
\begin{equation*}
	\sum_{b\in\Fp^3}\zeta^{b^{\mathsf T}K_ab}=p^2\sum_{z\in\Fp}\zeta^{-t z^2}=p^2g.
\end{equation*}
The exceptional level set $Q(a)=r$ has
\begin{equation}\label{eq:observable-shell}
	N_r:=\#\{a\in\Fp^3:Q(a)=r\}=p^2+sp
\end{equation}
points. For completeness, character orthogonality gives
\begin{equation*}
	N_r=\frac1p\sum_{h\in\Fp}\zeta^{-hr}\left(\sum_{z\in\Fp}\zeta^{hz^2}\right)^3=p^2+sp,
\end{equation*}
using that $r$ is a nonzero square and the preceding Gauss-sum identities.

As $\sum_a\zeta^{-Q(a)}=g^3$, splitting~\eqref{eq:observable-finite-sum}  into the cases $Q(a)\neq r$ and $Q(a)=r$ gives
\begin{align*}
	Z_u(M_p)&=p^{-3}g^3\left[g^3\bigl(g^3-N_r\zeta^{-r}\bigr)+p^2gN_r\zeta^{-r}\right]\\
	&=p^{-3}g^9+p^{-3}g^4(p^2-g^2)N_r\zeta^{-r}\\
	&=p g+p(p^2-1)\zeta^{-4u^{-2}},
\end{align*}
where $g^2=sp$ and $N_r=p(p+s)$ were used in the last line. Since $p(p^2-1)\ne0$ and $\zeta$ is primitive, $Z_u(M_p)=Z_v(M_p)$ is equivalent to $u^{-2}=v^{-2}$, hence to $u=\pm v$. Theorem~\ref{thm:step2} and Remark~\ref{rem:converse} identify these pairs with exactly the braided equivalence classes of the family.
\end{proof}

For $p=5$, using $G_5(-1)=\sqrt5$ gives
\begin{equation}\label{eq:observable-p5}
Z_1(M_5)=5\sqrt5+120\zeta,
\qquad
Z_2(M_5)=5\sqrt5+120\zeta^{-1}.
\end{equation}
These values are distinct and are complex conjugates. Thus the distinguishing observable is the complex amplitude on the specified oriented manifold. Taking only its absolute value would lose the distinction in this example.

\subsection{Physical meaning and comparison with two-component probes}

The finite sum has the usual finite-gauge-theory interpretation: $a,b,c\in V$ are the holonomies around the three surgery meridians, which generate $H_1(M_p,\mathbb Z)$, and the factor $p^{-3}=|V|^{-1}$ is the gauge-volume normalization~\cite[\S6.2, equations~(6.8)--(6.10)]{DW}. The quadratic terms in~\eqref{eq:observable-finite-sum} come from the $p$-fold twists, while the determinant comes from the alternating three-flux response of the Borromean link. The surgery sum combines these responses and sums over all colors, leaving an amplitude of the closed spacetime without an externally chosen identification of its anyons. Unlike the comparison of colored link data, this invariant requires no chosen bijection $\Lambda_{uv}$ between the simple objects of two categories.

There is a direct comparison with the preceding indistinguishability theorem. If $M$ admits an integral surgery presentation with at most two link components,~\eqref{eq:observable-surgery} and Theorem~\ref{thm:step1} imply that $Z_u(M)$ is independent of $u\in\Fp^\times$, since the bijections $\Lambda_{uv}$ preserve dimensions and every framed link evaluation in that sum. The three-component presentation above yields~\eqref{eq:observable-separates}: an analytically computable partition function distinguishing every inequivalent pair in this family, beyond the common one- and two-component RT link invariants.

\section*{Acknowledgments}
R.L. is supported by National Natural Science
Foundation of China under Grant No.~12422503. JT is supported by National Natural Science Foundation of China under Grant No.~12405085 and by the Natural Science Foundation of Shanghai (Grant No.~24ZR1419300).

\end{document}